\documentclass[a4paper,10pt]{article}
\usepackage{amssymb,amsmath,amsthm}
\usepackage{fullpage}
\usepackage{hyperref}
\usepackage{eucal}
\usepackage{color}
\usepackage{stackrel}
\usepackage{graphicx}
\usepackage[shortlabels]{enumitem}
\newcommand{\ie}{\emph{i.e.}}

\newcommand{\Real}{\mathbb{R}}

\newcommand{\Nat}{\mathbb{N}}

\newcommand{\Dom}{\mathsf{D}}

\newcommand{\curl}{\mathop{\mathrm{curl}}\nolimits}
\newcommand{\eps}{\varepsilon}
\newcommand{\sii}{L^2}

\newcommand{\der}{\mathrm{d}}

\newcommand{\demi}{\mbox{$\frac{1}{2}$}}
\newtheorem{Theorem}{Theorem}
\newtheorem{Lemma}{Lemma}
\newtheorem{Proposition}{Proposition}

\newtheorem{Conjecture}{Conjecture}
\theoremstyle{definition}
\newtheorem{Remark}{Remark}

\def\OMIT#1{}
\usepackage[normalem]{ulem}
\definecolor{DarkGreen}{rgb}{0,0.5,0.1} 

\newcommand\soutD{\bgroup\markoverwith
{\textcolor{DarkGreen}{\rule[.5ex]{2pt}{1pt}}}\ULon}
\newcommand\soutP{\bgroup\markoverwith
{\textcolor{blue}{\rule[.5ex]{2pt}{1pt}}}\ULon}
\newcommand{\Hm}[1]{\leavevmode{\marginpar{\tiny%
$\hbox to 0mm{\hspace*{-0.5mm}$\leftarrow$\hss}%
\vcenter{\vrule depth 0.1mm height 0.1mm width \the\marginparwidth}%
\hbox to
0mm{\hss$\rightarrow$\hspace*{-0.5mm}}$\\\relax\raggedright #1}}}

\begin{document}
%
\title{\textbf{\LARGE Is the optimal magnetic rectangle a square?}}
\author{David Krej\v{c}i\v{r}{\'\i}k}
\date{\small 
\emph{
\begin{quote}
\begin{center}
Department of Mathematics, Faculty of Nuclear Sciences and 
Physical Engineering, Czech Technical University in Prague, 
Trojanova 13, 12000 Prague 2, Czech Republic;
david.krejcirik@fjfi.cvut.cz
\end{center}
\end{quote}
}
\smallskip
19 August 2025}
\maketitle
 
%
\begin{abstract}
\noindent
We are concerned with the dependence of the lowest eigenvalue of the magnetic Dirichlet Laplacian on the geometry of rectangles,
subject to homogeneous fields. 
We conjecture that the square is a global minimiser 
both under the area or perimeter constraints.
Contrary to the well-known magnetic-free analogue,
the present spectral problem does not admit explicit solutions.
By establishing lower and upper bound to the eigenvalue,
we establish the conjecture for weak magnetic fields.
Moreover, we relate the validity of the conjecture
to the simplicity of the eigenvalue
and symmetries of minimisers of
a non-convex minimisation problem.
%
%
%
\end{abstract}

\section{Introduction}
This paper is about how little we know
and how much the current mathematical tools are limited. 

Among all membranes of a given area, the circular one
produces the lowest fundamental tone. 
This is a well-known interpretation of the celebrated 
Faber--Krahn inequality stating that 
\begin{equation}\label{FK}
  \lambda_1(\Omega) \geq \lambda_1(\Omega^*)
  \,,
\end{equation}
where~$\Omega$ is any bounded open planar set, 
$\Omega^*$ is the disk of the same area
and $\lambda_1(\Omega)$ is the lowest eigenvalue of 
the boundary value problem 
\begin{equation}\label{Dirichlet} 
\left\{
\begin{aligned}
  -\Delta u &= \lambda u
  && \mbox{in} && \Omega 
  \,,
  \\
  u &= 0 
  && \mbox{on} && \partial\Omega 
  \,.
\end{aligned}
\right.
\end{equation}

Restricting to rectangles,
it is also true that the square is the optimal geometry
for the Dirichlet Laplacian.
More specifically, defining 
\begin{equation}\label{rectangle} 
  \Omega_{a,b} := 
  \left(-\mbox{$\frac{a}{2}$},\mbox{$\frac{a}{2}$}\right) 
  \times \left(-\mbox{$\frac{b}{2}$},\mbox{$\frac{b}{2}$}\right) 
  ,
\end{equation}
where $a,b$ are any positive numbers,
the inequality~\eqref{FK} remains true 
for any rectangle~$\Omega_{a,b}$ 
instead of the arbitrary domain~$\Omega$
and a square $\Omega_{a,b}^*:=\Omega_{a,a}$ 
instead of the disk~$\Omega^*$
whenever the area $|\Omega_{a,b}|=ab$ is fixed.
While the general proof based on symmetrisation techniques 
applies to arbitrary quadrilaterals \cite[Sec.~3.3]{Henrot},
the case of rectangles can alternatively be established
in an elementary way just by using the well-known fact that 
the problem~\eqref{Dirichlet} is explicitly solvable
by separation of variables in terms of sine and cosine functions.  
We refer to~\cite{Laugesen_2019} 
for a recent spectral optimisation of the Laplacian eigenvalues  
in the larger generality of rectangular boxes
with Robin boundary conditions.

Interpreting the Laplacian as the free Schr\"odinger operator
and the Dirichlet condition as hard-wall boundaries,
the result~\eqref{FK} equally  says that the ground-state energy
of a non-relativistic quantum particle constrained 
to nanostructures of a given material 
is minimised by the disk. 
It has been conjectured recently that~\eqref{FK} also holds
in the relativistic setting of the Dirac operator 
with infinite-mass boundary conditions 
instead of the Dirichlet Laplacian,
see~\cite{Antunes-Benguria-Lotoreichik-Ourmires-Bonafos_2021}
(respectively, \cite{BK4}) for arbitrary domains 
(respectively, rectangles).
Despite a strong numerical evidence 
\cite{Antunes-Benguria-Lotoreichik-Ourmires-Bonafos_2021}
(respectively, \cite{ABK})
and partial analytical results
\cite{Benguria-Fournais-Stockmeyer-Bosch_2017,
Lotoreichik-Ourmieres_2019,
Antunes-Benguria-Lotoreichik-Ourmires-Bonafos_2021}
(respectively, \cite{BK4,ABK}),
however, the proof of the relativistic conjectures
seems to be beyond the reach of current mathematical apparatus.
Indeed, no symmetrisation techniques are available for the Dirac operator. 
What is more, the case of rectangles is just illusively simpler,
for the Dirac problem cannot be solved by separation of variables.

The objective of this paper is to point out yet another 
spectral optimisation problem which remains a mystery
in the case of rectangles.
Physically, it is still about the non-relativistic quantum particle
constrained to~$\Omega$, 
but now subject to homogeneous magnetic fields.
The ground-state energy of this system coincides with
the lowest eigenvalue $\lambda_1^B(\Omega)$ of 
the boundary value problem
\begin{equation}\label{Dirichlet.magnet} 
\left\{
\begin{aligned}
  (-i\nabla-A)^2 u &= \lambda u
  && \mbox{in} && \Omega 
  \,,
  \\
  u &= 0 
  && \mbox{on} && \partial\Omega 
  \,,
\end{aligned}
\right.
\end{equation}
where $A:\overline{\Omega}\to\Real^2$ is any smooth vector potential
generating a constant magnetic field $\curl A =: B \in \Real$.
By the Erd\H{o}s--Faber--Krahn inequality~\cite{Erdos_1996}
(see also~\cite{Ghanta-Junge-Morin_2024}),
one still has the magnetic variant of~\eqref{FK}:
\begin{equation}\label{EFK}
  \lambda_1^B(\Omega) \geq \lambda_1^B(\Omega^*)
  \,,
\end{equation}
valid for every open set $\Omega \subset \Real^2$
of fixed area $|\Omega| = |\Omega^*|$
and any $B \in \Real$.
Analogous (reversed) inequalities have been recently studied
for the principal eigenvalue of the magnetic Neumann Laplacian
\cite{Kachmar-Lotoreichik_2024,Colbois-Lena-Provenzano-Savo_2024}.
In the case of rectangles, 
the following conjecture remains open.

\begin{Conjecture}\label{Conj.main}
For every $a>0$ and $B \in \Real$,
\begin{equation}\label{iso}
  \lambda_1^B(\Omega_{a,a^{-1}}) 
  \geq \lambda_1^B(\Omega_{1,1})
  \,.
\end{equation}
\end{Conjecture}

Here we consider the class of rectangles of area equal to~$1$,
but there is no loss of generality in this restriction,
for other values can be recovered by scaling.

\begin{Remark}
Following~\cite{BK4}, it is also possible to state an analogous 
conjecture with a fixed perimeter instead of the fixed area. 
However, the isoperimetric constraint is known to be easier.
Indeed, it is implied by Conjecture~\ref{Conj.main}
via scaling and the classical (geometric) isoperimetric inequality.
\end{Remark}

We emphasise that the general validity of Conjecture~\ref{Conj.main}
if far from being obvious.
Indeed, one gets easily convinced that 
the problem~\eqref{Dirichlet.magnet} 
with~$\Omega_{a,b}$ instead of~$\Omega$
cannot be solved by separation of variables.
At the same time, no suitable symmetrisation 
techniques are available for complex-valued functions.
In this paper, we show that Conjecture~\ref{Conj.main} 
holds in the regime of weak magnetic fields.

\begin{Theorem}\label{Thm.small}
There exists a positive constant~$C$ (independent of~$a$)
such that~\eqref{iso} holds for every $|B| \leq C$.
\end{Theorem}

The organisation of this paper is as follows.
In Section~\ref{Sec.pre} we settle the problem 
in terms of a spectral problem of a self-adjoint operator
in a fixed (\ie~$a$-independent) Hilbert space.
Section~\ref{Sec.symmetry} is not needed for the proof
of Theorem~\ref{Thm.small}, but it offers a robust way
how to establish Conjecture~\ref{Conj.main}
as a consequence of a symmetry of a non-convex optimisation.
In Section~\ref{Sec.local} we show that the square is
a local minimiser for weak magnetic fields., 
establishing thus Conjecture~\ref{Conj.main}
for rectangles close to the square.
In Section~\ref{Sec.quantitative} we establish 
upper and lower bounds to the eigenvalue
$\lambda_1^B(\Omega_{a,a^{-1}})$,
obtaining thus Conjecture~\ref{Conj.main}
for rectangles far from the square.

\section{Preliminaries}\label{Sec.pre}
The boundary value problem~\eqref{Dirichlet.magnet} 
in the rectangle $\Omega_{a,a^{-1}} =: \Omega_a$ 
is understood as the spectral problem for the self-adjoint
operator~$H_a^A$ in $\sii(\Omega_a)$
associated with the form
\begin{equation} 
  h_a^A[u] := \|\partial_1^A u\|^2 + \|\partial_2^A u\|^2 
  \,, 
  \qquad
  u \in \Dom(h_a^A) := W_0^{1,2}(\Omega_a)
  \,,
\end{equation}
where~$\|\cdot\|$ denotes the norm of $\sii(\Omega_a)$ 
and we abbreviate $\partial_k^A := \partial_k-iA_j$
with $k \in \{1,2\}$.

The spectrum of~$H_a^A$ is independent of the choice
of the vector potential~$A$ giving 
the same magnetic field $\curl A = B$.
Indeed, if $\tilde{A}:\overline{\Omega_a} \to \Real^2$
is another smooth potential satisfying $\curl \tilde{A} = B$,
then $\curl(\tilde{A}-A)=0$, so there exists a smooth function 
$\phi:\overline{\Omega_a} \to \Real$ satisfying
$\tilde{A}-A = \nabla \phi$.
Consequently, $H_a^{\tilde{A}} = e^{i\phi} H_a^A e^{-i\phi}$,
which means that the operators $H_a^{\tilde{A}}$ and $H_a^A$
are unitarily equivalent, therefore isospectral.
This is the well-known gauge invariance.
Without loss of generality,
we choose the gauge
\begin{equation}\label{gauge}
  A(x) := \big(-\theta x_2,(1-\theta)x_1\big) B
  \,,
\end{equation}
where $x \in \Omega_{a}$
and $\theta,B \in \Real$ are arbitrary.

It is useful to work in a Hilbert space
independent of the parameter~$a$.
More specifically, we introduce the unitary transform
$
  U: \sii(\Omega_a)
  \to \sii(\Omega_1)
$
by setting $(Uu)(x):=u(a x_1,a^{-1} x_2)$
and define a unitarily equivalent (therefore isospectral) operator
$\hat{H}_{a}^A := U H_{a}^A U^{-1}$.
It is associated with the form~$\hat{h}_a^A$ defined by
$\hat{h}_a^A[u] := h_a^A[U^{-1}u]$ 
and $\Dom(\hat{h}_a^A) := U \Dom(h_a^A)$. 
It is easily verified that
\begin{equation} 
  \hat{h}_a^A[u] = 
  a^{-2} \, \|\partial_1^A u\|^2 
  + a^2 \, \|\partial_2^A u\|^2 
  \,, 
  \qquad
  u \in \Dom(\hat{h}_a^A) = W_0^{1,2}(\Omega_1)
  \,,
\end{equation}
where~$\|\cdot\|$ denotes the norm of $\sii(\Omega_1)$
and~$A$ still means~\eqref{gauge} but with $x \in \Omega_1$ now.  
Clearly, $\hat{H}_{1}^A = H_{1}^A$.
Moreover, $\Dom(\hat{h}_{a}^A)=\Dom(\hat{h}_{1}^A)$ for every $a>0$,
so we actually have $\Dom(\hat{h}_{a}^A) = \Dom(h_{1}^A)$.

A variational characterisation of the lowest eigenvalue
$\lambda_1^B(\Omega_a)$ reads
\begin{equation}\label{Rayleigh} 
  \lambda_1^B(\Omega_a)
  = \inf_{\stackrel[u \not= 0]{}{u \in W_0^{1,2}(\Omega_1) }} 
  \frac{\hat{h}_a^A[u]}{\, \|u\|^2}  
  \,,
\end{equation}
where the infimum can be replaced by a minimum.

\section{From symmetry to optimality}\label{Sec.symmetry}
In this section, we explain how the robust idea of~\cite{BK4} 
suggested to prove the optimality of a square among all Dirac rectangles
can be adapted to the present magnetic setting.

Using the inequality between the arithmetic and geometric means,
\eqref{Rayleigh}~implies 
\begin{equation}\label{area}
  \lambda_1^B(\Omega_a)
  \geq \inf_{\stackrel[u \not= 0]{}{u \in W_0^{1,2}(\Omega_1) }} 
  \frac{2 \, \|\partial_1^A u\| \, \|\partial_2^A u\|}{\, \|u\|^2}  
  =: \mu
  \,,
\end{equation}
The minimisation problem on the right-hand side of~\eqref{area}
does not involve a convex functional.
In fact, the associated Euler equation is a non-linear problem.
Recalling that we use the symbol $\|\cdot\|$
for the norm of $\sii(\Omega_1)$,
let us denote by $(\cdot,\cdot)$ the corresponding inner product.

\begin{Lemma}\label{Lem.minimiser}
The infimum on the right-hand side of~\eqref{area} is achieved.
Any minimiser~$u$ satisfies 
the weak eigenvalue equation 
\begin{equation}\label{Euler}
  \alpha^{-2} \, (\partial_1^A v,\partial_1^A u)
  + \alpha^2 \, (\partial_2^A v,\partial_2^A u)
  = \mu \, (v,u)
\end{equation} 
for every $v \in W_0^{1,2}(\Omega_1)$,
where
$$
  \alpha := \sqrt{\frac{\|\partial_1^A u\|}{\|\partial_2^A u\|}}
  \,.
$$
\end{Lemma}
\begin{proof}
First of all, 
let us argue that 
the infimum in~\eqref{area} is indeed achieved.
Define the functional
\begin{equation*} 
  J[u] := 2\,\|\partial_1^A u\| \, \|\partial_2^A u\|
  \,, \qquad
  u \in \Dom(J) := W_0^{1,2}(\Omega_1)
  \,.
\end{equation*}
Then 
\begin{equation}\label{area.bis}  
  \mu = \inf_{\stackrel[\|u\| = 1]{}
  {u \in W_0^{1,2}(\Omega_1)}} 
  J[u] 
  \,.
\end{equation}
The functional~$J$ is coercive and lower semicontinuous.
\begin{itemize}[left=0ex,itemsep=0ex]
\item
To see that~$J$ is coercive, we observe that 
the diamagnetic and Poincar\'e inequalities imply
$
  \|\partial_k^A u\| \geq \|\partial_k |u|\| \geq \pi \|u\|
$
for every $u \in W_0^{1,2}(\Omega_1)$ and $k \in \{1,2\}$.
At the same time, 
$
  \|\partial_k^A u\|^2 
  \geq \frac{1}{2} \|\partial_k u\|^2 - \|A_k\|_\infty^2 \|u\|^2
$.
These two inequalities can be combined in an elementary way
to show that there exists a positive constant~$c$ such that
$J[u] \geq c \, \|u\|_{W^{1,2}(\Omega_1)}$. 
\item
To see that~$J$ is lower semicontinuous, we use the facts that
the product of two lower semicontinuous functions 
is lower semicontinuous and that the square root of 
a lower semicontinuous function is lower semicontinuous.
So, it remains to show that 
$u \mapsto \|\partial_1^A u\|^2$ and $u \mapsto \|\partial_2^A u\|^2$
are lower semicontinuous on $L^2(\Omega_1)$,
with the convention that the action is~$+\infty$
if $u \not\in W_0^{1,2}(\Omega_1)$.  
Let us focus on the former, the proof for the latter is analogous.
For every $n \in \Nat$, define
$$
  h_n[u] := \left(
  u,n \, (-\partial_1^A)^2 \, [n+(-\partial_1^A)^2]^{-1} u
  \right)
  \,, \qquad
  \Dom[h_n] := L^2(\Omega_1)
  \,,
$$
where $(-\partial_1^A)^2$ is understood as a self-adjoint realisation
of this operator in $\sii((-\demi.\demi))$,
subject to Dirichlet boundary conditions,
and we denote by the same symbol the operator 
$(-\partial_1^A)^2 \otimes 1$ in $\sii(\Omega_1)$.
Then~$h_n$ is bounded on the unit ball of $L^2(\Omega_1)$
and continuous.
By using the spectral theorem, 
$h_n[u]$ increases monotonically to $h[u] := \|\partial_1^A u\|^2$
for every $u \in \sii(\Omega_1)$
(recall that $h[u]=+\infty$ if $u \not\in W_0^{1,2}(\Omega_1)$).
This implies that~$h$ is lower semicontinuous.
\end{itemize}

\noindent
Let $\{u_j\}_{j \in \Nat}$ be a minimising sequence,
\ie\ $J[u_j] \to \mu$ as $j \to \infty$ 
and $\|u_j\|=1$ for every $j \in \Nat$.
Consequently,  
\begin{equation*}
  c \, \|\nabla u_j\|
  \leq J[u_j] = \mu
  \,.
\end{equation*}
It follows that $\{u_j\}_{j \in \Nat}$
is a bounded sequence in $W^{1,2}(\Omega_1)$.
Therefore, up to a subsequence, 
$\{u_j\}_{j \in \Nat}$ converges weakly to some~$\psi$ 
in $W^{1,2}(\Omega_1)$.
By the compactness of the embedding
$W^{1,2}(\Omega_1)$ in $\sii(\Omega_1)$,
we may assume that $\{u_j\}_{j \in \Nat}$ converges (strongly) 
to some~$u$ in $\sii(\Omega_1)$
such that $\|u\|=1$.
By using~$u$ as a trial function in~\eqref{area.bis}, 
we obviously have $\mu \leq J[u]$. 
On the other hand, 
\begin{equation*}
  \mu = \liminf_{j\to \infty} J[u_j] 
  \geq J[u]
  \,,
\end{equation*}
where the inequality follows by the property that~$J$
is lower semicontinuous.
In summary, $\mu = J[u]$, 
so the infimum in~\eqref{area.bis} can be replaced by a minimum.

Now, let~$u$ be any minimiser of~\eqref{area.bis}.
Then~$u$ is a critical point of the functional~$J$
and the derivative  
$$
  \lim_{\eps \to 0}
  \frac{1}{\eps}
  \left(
  \frac{J[u+\eps v]}{\|u+\eps v\|^2}
  - \frac{J[u]}{\|u\|^2}
  \right)
$$
is necessarily equal to zero for any choice 
of the test function $v \in W_0^{1,2}(\Omega_1)$.
This leads to the equation
\begin{equation*}
  \alpha^{-2} \, \Re(\partial_1^A v,\partial_1^A u)
  + \alpha^2 \, \Re(\partial_2^A v,\partial_2^A u)
  = \mu \, \Re(v,u)
  \,.
\end{equation*}
Combining this equation with its variant where~$v$ is replaced by $iv$,
it is clear that the real part can be removed,
so~\eqref{Euler} follows. 
 
Finally, let us argue that~\eqref{Euler} is well defined,
meaning that~$\alpha$ is positive and bounded. 
If $\|\partial_1^A u\|=0$, 
then the diamagnetic inequality implies $\|\partial_1 |u| \|=0$, 
so~$|u|$ is independent of the first variable,
which is incompatible with $|u| \in W_0^{1,2}(\Omega_1))$
unless $|u| = 0$ identically. 
An analogous argument excludes the possibility $\|\partial_2^A u\|=0$. 
\end{proof}

The following symmetry is naturally expected 
for the symmetric gauge~\eqref{gauge} with $\theta=\frac{1}{2}$.
\begin{Conjecture}\label{Conj.symmetry}
If $\theta=\frac{1}{2}$,
then there exists a minimiser~$u$
of the right-hand side of~\eqref{area}
which satisfies
\begin{equation}\label{symmetry}
  \|\partial_1^A u\| = \|\partial_2^A u\|
  \,.
\end{equation}
\end{Conjecture}

Note that this conjecture is true provided that 
the minimisation problem on the right-hand side of~\eqref{area}
admits a \emph{unique} minimiser.
Indeed, it is easy to see that if~$u$ is a minimiser,
then so is the rotated function $x \mapsto u(-x_2,x_1) =: v(x)$. 
The uniqueness of the minimiser together with the identities
$\|\partial_1^A v\| = \|\partial_2^A u\|$
and $\|\partial_2^A v\| = \|\partial_1^A u\|$
then implies~\eqref{symmetry}.

Because of the non-linear structure of the minimisation problem
and the lack of positivity preserving property 
for the magnetic Laplacian,
we have been able to establish 
neither the uniqueness of the minimiser 
nor Conjecture~\ref{Conj.symmetry}, respectively.
This is unfortunate, because its validity  
immediately implies Conjecture~\ref{Conj.main}.

\begin{Theorem}\label{Thm.idea}
Conjecture~\ref{Conj.symmetry} implies Conjecture~\ref{Conj.main}.
\end{Theorem}
\begin{proof}
Let us assume $\theta=\frac{1}{2}$.
As a consequence of~\eqref{area} 
and Conjecture~\ref{Conj.symmetry},
\begin{equation}\label{comp1}  
  \lambda_1^B(\Omega_a)
  \geq \inf_{\stackrel[u \not= 0]{}
  {u \in W_0^{1,2}(\Omega_1) \ \& \ \eqref{symmetry} \text{ holds}}} 
  \frac{2\,\|\partial_1^A u\| \, \|\partial_2^A u\|}
  {\, \|u\|^2}  
\end{equation}
for every $a>0$ and $B \in \Real$.
At the same time, because of the rotational symmetry 
of the square~$\Omega_1$,
it is easy to see that there exists an eigenfunction~$u$ of~$\hat{H}_1$
satisfying~\eqref{symmetry}. 
Consequently,
\begin{equation}\label{comp2} 
  \lambda_1^B(\Omega_1)
  = \inf_{\stackrel[u \not= 0]{}
  {\psi \in W_0^{1,2}(\Omega_1) \ \& \ \eqref{symmetry} \text{ holds}}} 
  \frac{\|\partial_1^A u\|^2 + \|\partial_2^A u\|^2}
  {\, \|u\|^2}  
  = \inf_{\stackrel[u \not= 0]{}
  {\psi \in W_0^{1,2}(\Omega_1) \ \& \ \eqref{symmetry} \text{ holds}}} 
  \frac{2\,\|\partial_1^A u\| \, \|\partial_2^A u\|}
  {\, \|u\|^2}  
  \,.
\end{equation}
Comparing~\eqref{comp1} and~\eqref{comp2}, 
we get Conjecture~\ref{Conj.main}.
\end{proof}

\section{From simplicity to local optimality}\label{Sec.local}
Unable to prove Conjecture~\ref{Conj.main} in its full generality,
in this section we focus on showing that the square~$\Omega_1$
is at least a local \emph{minimiser} among all rectangles~$\Omega_a$
with $a>0$. 

For simplicity, we abbreviate $\lambda_a := \lambda_1^B(\Omega_a)$.

\begin{Lemma}\label{Lem.derivatives}
Assume $\theta=\frac{1}{2}$.
Let $B \in \Real$ be such that
the eigenvalue $\lambda_1$ is simple. 
Then
\begin{align} 
  \left.
  \frac{\partial \lambda_1^B(\Omega_a)}{\partial a}
  \right|_{a=1} &= 0
  \,,
  \label{derivatives1}
  \\
  \frac{1}{2}
  \left.
  \frac{\partial^2 \lambda_1^B(\Omega_a)}{\partial a^2}
  \right|_{a=1} &= 2 \, \lambda_1^B(\Omega_1)
  + \lambda_1^B(\Omega_1) \, \|\dot{u}_1\|^2
  - \|\partial_1^A \dot{u}_1\|^2 - \|\partial_2^A \dot{u}_1\|^2
  \,,
  \label{derivatives2}
\end{align}
where~$\dot{u}_1$ is the solution of the problem 
\begin{equation}\label{weak.dot} 
  (\partial_1^A v,\partial_1^A \dot{u}_1) 
  + (\partial_2^A v,\partial_2^A \dot{u}_1)
  - \lambda_1^B(\Omega_1) \, (v,\dot{u}_1)
  =
  2 \, (\partial_1^A v,\partial_1^A u_1) 
  - 2 \, (\partial_2^A v,\partial_2^A u_1) 
\end{equation}
for every $v \in W_0^{1,2}(\Omega_1)$,
with~$u_1$ being the eigenfunction of~$\hat{H}_1$
corresponding to $\lambda_1^B(\Omega_1)$ and normalised by $\|u_1\|=1$.
\end{Lemma}
\begin{proof}
It is straightforward to verify that 
$\{\hat{h}_a^A\}_{a > 0}$ 
is a holomorphic family of forms of type~(a) in the sense of Kato 
\cite[Sec.~VII.4]{Kato}. 
Indeed, one can directly use the criterion \cite[Sec.~VII.4.8]{Kato}.
Consequently, $\{\hat{H}_a^A\}_{a > 0}$ 
is a holomorphic family of operators of type~(B).
Because of the simplicity hypothesis,
$a \mapsto \lambda_1^B(\Omega_a) =: \lambda_a$
is a real-analytic function on a neighbourhood of~$a=1$.
At the same time, the corresponding eigenfunction~$u_a$ 
satisfying the normalisation $\|u_a\|=1$ 
is a real-analytic function in the topology of $W_0^{1,2}(\Omega_1)$
on a neighbourhood of~$a=1$.
Then the claims follow by a routine differentiation 
of the weak formulation of the eigenvalue equation
\begin{equation}\label{weak}  
  a^{-2} \, (\partial_1^A v,\partial_1^A u_a) 
  + a^2 \, (\partial_2^A v,\partial_2^A u_a)
  = \lambda_a \, (v,u_a)
\end{equation}
for every $v \in W_0^{1,2}(\Omega_1)$.

In detail,  
differentiating~\eqref{weak} with respect to~$a$
and denoting the corresponding derivative by a dot, we find 
\begin{equation}\label{first}  
  a^{-2} \, (\partial_1^A v,\partial_1^A \dot{u}_a) 
  + a^2 \, (\partial_2^A v,\partial_2^A \dot{u}_a)
  - 2 a^{-3} \, (\partial_1^A v,\partial_1^A u_a) 
  + 2a \, (\partial_2^A v,\partial_2^A u_a) 
  = \lambda_a \, (v,\dot{u}_a)
  + \dot\lambda_a \, (v,u_a)
  \,.
\end{equation}
This equation reduces to~\eqref{weak.dot} for $a=1$.
Taking $v=u_\alpha$ in~\eqref{first},
$v = \dot{u}_\alpha$ in~\eqref{weak}     
and combining these two identities evaluated at $a=1$,
we find
\begin{equation}\label{first.derivative}
  \dot\lambda_1 = -2 \, \|\partial_1^A u_1\|^2
  + 2 \, \|\partial_2^A u_1\|^2
  = 0
  \,,
\end{equation}
where the second equality follows by the rotational symmetry 
of the square.
Indeed, $u_1$~necessarily satisfies~\eqref{symmetry}
with~$u_1$ instead of~$u$.  
This establishes~\eqref{derivatives1}. 

Differentiating~\eqref{first} with respect to~$a$, we find
\begin{multline}\label{second}
  a^{-2} \, (\partial_1^A v,\partial_1^A \ddot{u}_a) 
  + a^2 \, (\partial_2^A v,\partial_2^A \ddot{u}_a)
  - 4 a^{-3} \, (\partial_1^A v,\partial_1^A \dot{u}_a) 
  + 4a \, (\partial_2^A v,\partial_2^A \dot{u}_a) 
  \\
  + 6 a^{-4} \, (\partial_1^A v,\partial_1^A u_a) 
  + 2 \, (\partial_2^A v,\partial_2^A u_a) 
  = \lambda_a \, (v,\ddot{u}_a)
  + 2 \dot\lambda_a \, (v,\dot{u}_a)
  + \ddot\lambda_a \, (v,u_a)
  \,.
\end{multline}
Taking $v=u_\alpha$ in~\eqref{second},
$v = \ddot{u}_\alpha$ in~\eqref{weak},     
combining these two identities evaluated at $a=1$
and using~\eqref{first.derivative},
we find
$$
\begin{aligned}
  \ddot\lambda_1 
  &= 6 \, \|\partial_1^A u_1\|^2
  + 2 \, \|\partial_2^A u_1\|^2
  - 4 \, (\partial_1^A u_1,\partial_1^A \dot{u}_1) 
  + 4 \, (\partial_2^A u_1,\partial_2^A \dot{u}_1) 
  \,.
\end{aligned}  
$$
Using the symmetry and~\eqref{weak.dot},
we arrive at~\eqref{derivatives2}. 
\end{proof}

The desired local optimality of the square is equivalent to
showing that the second derivative of $\lambda_1^B(\Omega_a)$
with respect to~$a$ is positive.
However, it is not clear whether the right-hand side 
of~\eqref{derivatives2} is positive. 
We prove it for weak magnetic fields. 

\begin{Theorem}\label{Thm.local}
There exists a positive constant~$C$ such that,
for every $|B| \leq C$, 
\begin{equation}\label{derivatives} 
  \left.
  \frac{\partial \lambda_1^B(\Omega_a)}{\partial a}
  \right|_{a=1} = 0
  \qquad \mbox{and} \qquad 
  \left.
  \frac{\partial^2 \lambda_1^B(\Omega_a)}{\partial a^2}
  \right|_{a=1} > 0 
  \,.
\end{equation}
\end{Theorem}
\begin{proof}
Recall that we use the gauge~\eqref{gauge} 
explicitly depending on~$B$. 
It is straightforward to verify that 
$\{\hat{h}_a^A\}_{B \in \Real}$ 
is a holomorphic family of forms of type~(a).  
Again, one can directly use the criterion \cite[Sec.~VII.4.8]{Kato}.
Consequently, $\{\hat{H}_a^A\}_{B \in \Real}$ 
is a holomorphic family of operators of type~(B).
Since $\lambda_1^0(\Omega_a)$ is simple for every $a>0$, 
there exists a positive constant~$C$ such that 
$\lambda_1^B(\Omega_1)$ remains simple 
whenever $|B| \leq C$.
It follows that $B \mapsto \lambda_1^B(\Omega_1) =: \lambda^B$
is a real-analytic function on a neighbourhood of $B=0$.
At the same time, the corresponding eigenfunction~$u^B$ 
satisfying the normalisation $\|u^B\|=1$ 
is a real-analytic function in the topology of $W_0^{1,2}(\Omega_1)$
on a neighbourhood of $B=0$.

Assuming $\theta=\frac{1}{2}$,
Lemma~\ref{Lem.derivatives} implies the first identity 
of~\eqref{derivatives}.
However, the hypothesis $\theta=\frac{1}{2}$ is redundant,
because $\lambda_1^B(\Omega_a)$ is gauge invariant.

At the same time, under the hypothesis $\theta=\frac{1}{2}$,
Lemma~\ref{Lem.derivatives} implies the identity~\eqref{derivatives2}.
Here the first term on the right-hand side 
is positive for all sufficiently small~$|B|$;
explicitly, $\lambda^B \to 2\pi^2$ as $B \to 0$.
At the same time, the right-hand side of~\eqref{weak.dot} 
tends to zero as $B \to 0$.
Indeed, $u^0(x_1,x_2) = \varphi(x_1) \varphi(x_2)$,
where 
\begin{equation}\label{ground} 
  \varphi(x) := \sqrt{2} \, \cos(\pi x_1)
\end{equation}
is the first eigenfunction of the Dirichlet Laplacian
in $\sii((-\demi,\demi))$ normalised to~$1$,
which ensures that ${\partial_1^A}^2 u^0 = {\partial_2^A}^2 u^0$. 
Consequently,
$$
  \lambda_1^B(\Omega_1) \, \|\dot{u}_1\|^2
  - \|\partial_1^A \dot{u}_1\|^2 - \|\partial_2^A \dot{u}_1\|^2
  \to 0
$$
as $B \to 0$.
In summary, the right-hand side of~\eqref{derivatives2}
converges to~$4\pi^2$ as $B \to 0$.
It follows that there exists a positive constant~$C$
(possibly smaller than the previously chosen~$C$)
such that the right-hand side of~\eqref{derivatives2}
is positive for all $|B| \leq C$.
Again, the smallness of~$|B|$ and the positivity
must be independent of the choice of gauge.
\end{proof}

\section{Quantitative bounds}\label{Sec.quantitative}
To establish the global result of Theorem~\ref{Thm.small},
in this section we establish explicit upper and lower bounds 
to the eigenvalue $\lambda_1^B(\Omega_a)$.

Given $\beta \in \Real$, let $\nu(\beta)$ denote 
the lowest eigenvalue of the operator
$T_\beta := -\partial_x^2 + \beta^2 x^2$ in $\sii((-\demi,\demi))$,
subject to Dirichlet boundary conditions.
It is easy to see that 
\begin{equation}\label{nu.bounds} 
  \pi^2 \leq \nu(\beta) \leq \pi^2 + c \, \beta^2  
  \qquad \mbox{with} \qquad
  c := \int_{-\demi}^{\demi} x^2 \, \varphi(x)^2 \, \der x
  = \frac{1}{12} - \frac{1}{2\pi^2} \approx 0.03
  \,,
\end{equation}
where~$\varphi$ is given in~\eqref{ground}.
While these estimates are good for small~$|\beta|$,
we have $\nu(\beta) \sim |\beta|$ as $|\beta| \to \infty$. 
See Figure~\ref{Fig.harmonic} for the dependence
of~$\nu(\beta)$ on~$\beta$.
Let~$\varphi_\beta$ denote the positive eigenfunction of~$T_\beta$
corresponding to~$\nu(\beta)$ and normalised to~$1$ 
in $\sii((-\demi,\demi))$. 
Of course, $\varphi_0 = \varphi$.

\begin{figure}[h!]
\begin{center}
\includegraphics[width=0.7\textwidth]{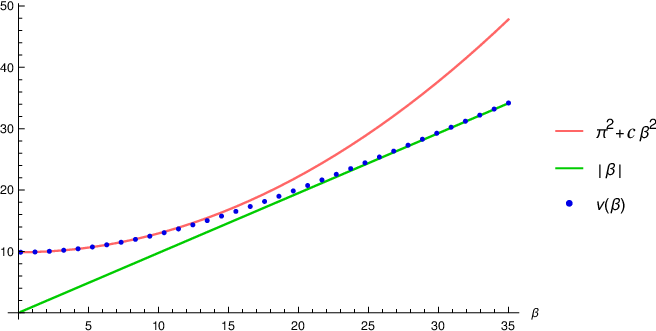} 
\end{center}
\caption{The eigenvalue~$\nu(\beta)$ as a function of~$\beta$
together with its asymptotics for small and large~$\beta$.}\label{Fig.harmonic}
\end{figure}

We start with an upper bound.
\begin{Proposition}\label{Prop.upper}
For every $a>0$ and $B \in \Real$,
\begin{equation}\label{upper} 
  \lambda_1^B(\Omega_a)
  \leq \pi^2 \, (a^{-2}+a^2)
  + c \, \frac{a^2}{1+a^4} \, B^2
  \,.
\end{equation}
\end{Proposition}
\begin{proof}
Using $x \mapsto \varphi_{\beta_1}(x_1) \varphi_{\beta_2}(x_2)$
with $\beta_1 := a^2(1-\theta) B$
and $\beta_1 := a^{-2}\theta B$
as a trial function in~\eqref{Rayleigh},
we get
\begin{align}
  \lambda_1^B(\Omega_a) 
  &\leq a^{-2} \, \nu\big(a^2(1-\theta) B\big) 
  + a^2 \, \nu\big(a^{-2}\theta B\big)
  \\
  &= (a^{-2} + a^2) \, \nu\left(\frac{a^2}{1+a^4} \, B \right) 
  \label{upper.better}
\end{align}  
where equality follows by the special choice $\theta := a^4/(1+a^4)$
(this is the best choice if the trial function
$x \mapsto \varphi_{0}(x_1) \varphi_{0}(x_2)$
were directly used instead). 
The announced bound~\eqref{upper} follows by~\eqref{nu.bounds}. 
\end{proof}

As for the lower bound, we have a trivial result,
which follows at once by the diamagnetic inequality.  
\begin{Proposition}\label{Prop.lower}
For every $a>0$ and $B \in \Real$,
\begin{equation}\label{lower} 
  \lambda_1^B(\Omega_a)
  \geq \pi^2 \, (a^{-2}+a^2)
  \,.
\end{equation}
\end{Proposition}
\begin{Remark} 
The bounds~\eqref{upper} and~\eqref{lower} are enough 
for our purposes of weak magnetic fields. 
For strong fields, a better upper bound is given by~\eqref{upper.better}.
Indeed, the inequality becomes sharp as $|B| \to \infty$.
A better lower bound for strong fields is given by 
the uniform bound $\lambda_1^B(\Omega_a) \geq |B|$.
It follows at once by the domain monotonicity,
by estimating~$H_a^A$ from below by the operator 
which acts in the same way but on the entire plane
(the Landau Hamiltonian).
Another lower bound, using the domain monotonicity 
$\Omega_{a,a^{-1}} \subset (-\frac{a}{2},\frac{a}{2})\times\Real$ 
with $\theta=0$ 
or $\Omega_{a,a^{-1}} \subset \Real \times (-\frac{a}{2},\frac{a}{2})$ 
with $\theta=1$, reads
$$
  \lambda_1^B(\Omega_a) 
  \geq \max\left\{
  a^{-2} \nu(a^2 B),a^2 \nu(a^{-2} B) 
  \right\}
  \,.
$$
\end{Remark}

Now we are in a position to establish Theorem~\ref{Thm.small}.
\begin{proof}[Proof of Theorem~\ref{Thm.small}]
Using Propositions~\ref{Prop.upper} (with $a=1$)
and~\ref{Prop.lower} (with arbitrary $a>0$),
we have 
\begin{equation} 
  \lambda_1^B(\Omega_a)
  \geq \pi^2 \, (a^{-2}+a^2) 
  \stackrel{!}{\geq}
  2 \pi^2 + \frac{c}{2} \, B^2
  \geq \lambda_1^B(\Omega_1)
  \,,
\end{equation}
where the central inequality (with~$!$) holds provided that
$$
  a^{-2}+a^2 \geq 2 + \frac{c}{2\pi^2} \, B^2 
  \,.
$$
In particular, this is satisfied if 
$$
  |a-1| \geq \sqrt{\frac{c}{2\pi^2}} \, |B| 
  \,.
$$
In summary, Conjecture~\ref{Conj.main} holds 
provided that~$|a-1|$ is sufficiently large,
with the largeness diminishing when $B \to 0$. 
To complete this argument for small values of~$|a-1|$,
we use Theorem~\ref{Thm.local}.
\end{proof}
%

%
\bibliography{bib}
\bibliographystyle{amsplain}

\end{document}